\documentclass[letterpaper,11pt]{amsart}
\usepackage[margin=1.2in]{geometry}
\usepackage{amsmath,amsthm,amssymb, mathtools}
\usepackage{thmtools, thm-restate}
\usepackage[T1]{fontenc}
\usepackage{enumerate}

\usepackage[dvipsnames]{xcolor}
\usepackage{hyperref, cleveref}
\hypersetup{
    colorlinks = true,
    linkcolor = BrickRed,
    citecolor = Green,
    urlcolor = Blue,
}
\usepackage{tikz}
\usepackage{float, caption}
\usepackage{parskip}
\usepackage{algorithm}
\usepackage{algpseudocodex}

\usepackage[
    backend=biber,
    style=alphabetic,
    sorting=ynt
]{biblatex}
\newcommand{\Z}{\mathbb Z}

\newcommand{\F}{\mathbb{F}}

\DeclareMathOperator{\im}{im}

\DeclareMathOperator{\rank}{rank}

\theoremstyle{plain} %theorems after this line but before the definition
\newtheorem{theorem}{Theorem}[section]

\newtheorem{proposition}[theorem]{Proposition} %the optional argument
\newtheorem{lemma}[theorem]{Lemma}

\theoremstyle{definition} %theorems after this line will all have the same
\newtheorem{definition}[theorem]{Definition} 
\newtheorem{notation}[theorem]{Notation} 
 
\newtheorem*{remark*}{Remark}

\title{On Lengths of $\mathbb{F}_2[x,y,z]/(x^{d_1}, y^{d_2},z^{d_3}, x+y+z)$}
\author{Han}
\address{University of Michigan} 
\email{fionahan@umich.edu}
\author{Kenkel}
\address{Grinnell College}
\email{kenkeljennifer@grinnell.edu}
\author{Li}
\address{University of Michigan} 
\email{dejl@umich.edu}
\author{Venkatesh}
\address{University of Michigan}
\email{srivenk@umich.edu}
\author{Wiles}
\address{University of Michigan} 
\email{ashwiles@umich.edu}

\date{\today}

\begin{document}

\begin{abstract} 
    In this paper, we provide a formula for the vector space dimension of the ring $\mathbb{F}_2[x,y,z]/(x^{d_1}, y^{d_2},z^{d_3}, x+y+z)$ over $\mathbb{F}_2$ when $d_1,d_2,d_3$ all lie between successive powers of $2$. This  recovers a result of Han in \cite{HAN92}, with a different proof. For general $d_1,d_2,d_3$, we provide a simple algorithm to calculate the vector space dimension of $\mathbb{F}_2[x,y,z]/(x^{d_1}, y^{d_2},z^{d_3}, x+y+z)$.
\end{abstract}

\maketitle

\newcommand{\J}{\mathcal{J}}
\newcommand{\R}{\mathbb{R}}
\newcommand{\ds}{\displaystyle}

\section{Introduction}
Let $R= \mathbb{F}[x_1, \dots, x_n]$ be the polynomial ring over a field $\F$ and let $I$ be an ideal such that $R/I$ is artinian, and thus a finite-dimensional vector space over $\F$. A natural question is to ask for the vector space dimension of $R/I$, both in total, and in each graded component.
%Perhaps the simplest family of ideals $I$ such that $R/I$ is artinian is when $I=(x_1^{d_1}, \dots, x_n^{d_n})$ for some non-negative integers $d_1, \dots, d_n$; the vector space dimensions of $R/I$ when $I$ is of this form were studied in \cite{RRR91}.
Perhaps the simplest example is $R/(x_1^{d_1}, \dots, x_n^{d_n})$ for some non-negative integers $d_1, \dots, d_n$; its total vector space dimension is just $\prod_{i=1}^n d_i$ while the vector space dimension in each graded component was studied in \cite{RRR91}.
Adding in only one ideal generator complicates the problem; the vector space dimensions of $R/I$ when $I=(x_1^{d_1}, \dots, x_n^{d_n}, x_1+\dots+x_n)$  are remarkably more complicated. Properties of this ring have been studied in \cite{HAN92,HM93,VRA15,syzOfIdeal,ResOfIdeal}, for example. 
In this paper, we study this ring in the special case $\mathbb{F}= \mathbb{F}_2 := \Z/2\Z$ and $n=3$.

%From this point on, we will denote $\Z/2\Z$ as $\mathbb{F}_2$. 

%The ring $\mathbb{F}[x_1, \dots, x_n]/(x_1^{d_1}, \dots, x_n^{d_n}, x_1+\dots+x_n)$ has been studied in various contexts. 
For a graded ring $S$, let $S_D$ denote the degree $D$ component. In the case that $\mathbb{F}$ is a field of characteristic 0, it was shown by Stanley \cite[Theorem 2.4]{WeylGroups} that the element $x+y+z + J$ has the Strong Lefschetz Property in the ring $\mathbb{F}[x,y,z]/(x^{d_1}, y^{d_2}, z^{d_3})$ meaning that the map 
$$\mathbb{F}[x,y,z]/(x^{d_1}, y^{d_2}, z^{d_3})_D \xrightarrow[]{x+y+z+J} \mathbb{F}[x,y,z]/(x^{d_1}, y^{d_2}, z^{d_3})_{D+1} $$
has maximal possible rank, where $J = (x^{d_1}, y^{d_2}, z^{d_3})$. In other words, multiplication by $x+y+z+J$ is either injective or surjective, depending on the vector space dimension of $\mathbb{F}[x,y,z]/(x^{d_1}, y^{d_2}, z^{d_3})$ in degrees $D$ and $D+1$. One could use this information to calculate $\dim_{\F}(R/(x^{d_1}, y^{d_2}, z^{d_3}, x+y+z))_D$.    
\par 
 
%and let $I$ be some $\mathfrak{m}$-primary ideal. Then $R/I$ is a finite dimensional $k$-vector space. A natural question is to determine what that finite dimension is. The total vector space dimension can be determined by totaling the vector space dimensions of each graded component. 
%For any graded ring $S$, we let $S_d$ denote the degree $d$ component. Let $R=k[x,y,z]$ and $I = (x^{d_1}, y^{d_2}, z^{d_3}, x+y+z)$ for some $d_1, d_2, d_3$ arbitrary natural numbers. 

\noindent However, the element $x+y+z+J$ does not have the Strong Lefschetz Property when $\mathbb{F}$ is a field of positive characteristic. As an example, consider $\mathbb{F}=\mathbb{F}_2$ and $d_1=d_2=d_3=2$ and $J=(x^2,y^2,z^2)$. Then $(\mathbb{F}_2[x,y,z]/((x^{2}, y^{2}, z^{2}))_1$ has vector space basis $\{x,y,z\}$ and thus is dimension 3, and $\mathbb{F}_2[x,y,z]/((x^{2}, y^{2}, z^{2})_2$ has vector space basis $\{xy, xz, yz\}$ and also has dimension 3. Yet multiplication by $x+y+z+J$ is not injective; note that $x+y+z$ is not in the ideal $(x^{2}, y^{2}, z^{2})$, yet $$(x+y+z)\cdot(x+y+z) = x^2+y^2+z^2,$$ which is in the ideal $(x^{2}, y^{2}, z^{2})$. Thus $(x+y+z+J) \cdot (x+y+z+J)= 0+J$ in the quotient ring $\mathbb{F}[x,y,z]/(x^{2}, y^{2}, z^{2})$, so multiplication by $x+y+z+J$ is not injective. 
 
\begin{notation}
    We write $\dim_\mathbb{F}(R/I)$ to denote the vector space dimension of $R/I$ over the field $\F$. Additionally, for convenience of notation, let
    \begin{align*}
        \dim_{\F}(d_1, d_2, d_3) & := \dim_{\F} \F[x, y, z]/(x^{d_1}, y^{d_2}, z^{d_3}, x+y+z).
    \end{align*}
\end{notation}
In \cite[Chapter 2]{HAN92}, it was shown that, when $\F$ is a field of positive characteristic, there exists a half-integer $[d_1,d_2,d_3]_{\F}$ such that
    $$
        \dim_\mathbb{F}(d_1,d_2,d_3) =
        \frac{2d_1d_2 + 2d_1d_3 + 2d_2d_3 - d_1^2 - d_2^2 - d_3^2}{4} + ([d_1,d_2,d_3]_\mathbb{F})^2.
    $$
Furthermore, the value of $[d_1,d_2,d_3]_\mathbb{F}$ is such that the generators of the colon ideal (see Definition \ref{colonideal} for the definition of a colon ideal)  $J=(x^{d_1}, y^{d_2}: (x+y)^{d_3})$ have degrees $\frac{d_1+d_2-d_3}{2} - [d_1,d_2,d_3]_{\F}$ and $\frac{d_1+d_2-d_3}{2} + [d_1,d_2,d_3]_{\F}$. The value of these ideal generators is not trivial to calculate; in \cite[Proposition 2.5]{HAN92}, she gives a relationship between $[d_1, d_2,d_3]_\mathbb{F}$ and a certain matrix of binomial coefficients. In \cite[Theorem 2.23]{HAN92}, she also gives a way of determining $[d_1,d_2,d_3]_{\F}$ depending on where the point $(d_1, d_2,d_3)$ lies in a geometric honeycombing of $\R^3$ with octahedra and tetrahedra. 

%Since we restrict our attention to $p=2$, we are able to be more explicit. 

% \begin{notation}
%     For convenience of notation, let
% \begin{align*}
%     \dim_{\mathbb{F}_2}(d_1, d_2, d_3) & := \dim_{\mathbb{F}_2}[x, y, z]/(x^{d_1}, y^{d_2}, (x+y)^{d_3})
% \end{align*}
% \end{notation}
% \DL{do we need this notation when we already have notation 1.1? pointing this out because the spacing is wacky (the table is not short enough to fit on this page)}

\noindent The goal of this paper is to obtain explicit formulas for $\dim_{\mathbb{F}}(d_1,d_2,d_3)$ when $\mathbb{F} = \mathbb{F}_2$. In the case that $d_1,d_2 \leq 2^e \leq d_3$, $\dim_{\mathbb{F}_2}(d_1,d_2,d_3)$ was given in \cite[Proposition 1.6(i)]{HAN92}, and the case that $d_1 \leq 2^e \leq d_2,d_3$, $\dim_{\mathbb{F}_2}(d_1, d_2, d_3)$ was described in \cite[Proposition 1.6(ii)]{HAN92} in terms of the dimension of a smaller vector space. The only remaining case is when $2^e \leq d_1 \leq d_2 \leq d_3 \leq 2^{e+1}$, the case in which all three powers are between two consecutive powers of 2; this situation is covered by \cite[Theorem 3.8]{HAN92}. 

Our main result is an alternate proof of the following formula for the dimension in the case that $2^e \leq d_1 \leq d_2 \leq d_3 \leq 2^{e+1}$.

\begin{restatable}{theorem}{majorThm} 
\label{theorem:formula-for-dimension}
    Let $q$ be a power of $2$ and let $\frac{q}{2}< d_1, d_2, d_3 \leq q$. Then:
    \begin{equation*} \label{equation:formula-for-dimension}
        \dim_{\mathbb{F}_2} \frac{\mathbb{F}_2[x,y,z]}{(x^{d_1},y^{d_2},z^{d_3},x+y+z)} = d_1d_2 + d_2d_3 + d_1d_3 -q(d_1+d_2+d_3) +q^2.
    \end{equation*}
\end{restatable}

To see how Theorem \ref{theorem:formula-for-dimension} follows from \cite[Theorem 3.8]{HAN92}, suppose that $q$ is a power of $2$ and that $\frac{q}{2}< d_1, d_2, d_3 \leq q$. Then $q-d_i< q-\frac{q}{2} = \frac{q}{2}$. Therefore, by \ref{corollary:han-result-power-of-2}, $\dim_{F_2}(d_1,q-d_2,q-d_3) = (q-d_2)(q-d_3)$. Thus \begin{align*} 
\dim_{\mathbb{F}_2}(d_1,d_2,d_3)&=	\dim_{\mathbb{F}_2}(d_1,q-d_2,q-d_3)+(d_2+d_3-q)d_1\\  &= (q-d_2)(q-d_3) + (d_2+d_3-q)d_1 \\ 
	&= q^2- qd_2 -qd_3 + d_2d_3 + d_1d_2+d_1d_3-q d_1 \\ 
	&= d_1d_2+d_2d_3+d_1d_3 - q(d_1+d_2+d_3) + q^2
	\end{align*} 
as desired.

We will now say a few words on how we came up with the formula. Using Macaulay2 (\cite{M2}), we first restricted our attention to the case when $d_2 = d_3$. To seek a pattern, we compared $\dim_{\mathbb{F}_2}(d_1, d_2, d_2)$ to $\dim_{\mathbb{F}_2}(d_1, d_2-1, d_2-1)$ and looked at values of $d_1$ ranging from 1 to 128, and $d_2$ ranging from 1 to 64. Part of the data is shown in \Cref{fig:macaulay2-data}.
\begin{figure}[ht!]
    \centering
    \captionsetup{justification=centering}
    \includegraphics[scale=0.5]{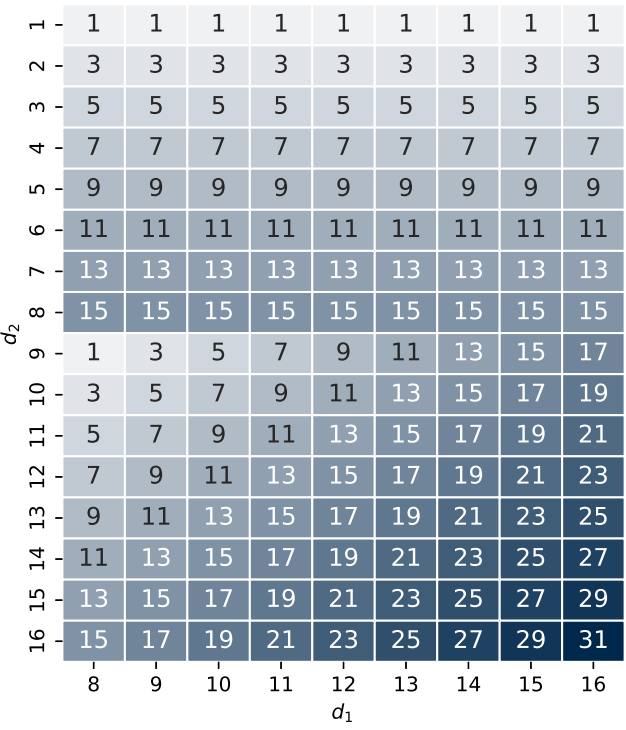}
    \vspace{-0.5\baselineskip}
    \caption{Column $i$, row $j$ stores the value: $\dim_{\mathbb{F}_2}(i, j, j) - \dim_{\mathbb{F}_2}(i, j-1, j-1)$} \label{fig:macaulay2-data}
\end{figure}
%\vspace{-10pt}

\noindent Notice that the value in each row in \Cref{fig:macaulay2-data} increases by 2 up until the row is a power of 2. Furthermore, for a fixed column, the data repeats after the row is larger than some power of 2. It follows that to find $\dim_{\mathbb{F}_2}(i,j,j)$, we could add up every entry in column $i$ (note also that $\dim_{\mathbb{F}_2}(i,0,0) = 0$). Using the formula for the sum of consecutive odds, and by observing how much the pattern deviates (see rows 8 and 9) after the next immediate power of 2, we were able to guess a formula for $\dim_{\mathbb{F}_2}(d_1,d_2,d_3)$. Finally, by fixing $d_1$ and varying $d_2, d_3$ all within the consecutive powers of 2, we were able to come up with a conjectural formula for $\dim_{\mathbb{F}_2}(d_1,d_2,d_3)$ when $\frac{q}{2} < d_1, d_2, d_3 \leq q$ for some $q = 2^e$, which we then proved by rather elementary methods.

It is unclear how well these results can be generalized to different ideals, as general ideals will not exhibit symmetry in the variables like the ideal $(x^{d_1}, y^{d_2}, z^{d_3}, x+y+z)$. Interestingly, a similar pattern was observed for the case of polynomials over $\F_3$, however it appears to be much more complex.

%It is unknown if a similar pattern occurs in characteristic 0 fields.
%\Jenny{See Reid, Roberts and Roitman, theorem 10. Also, Stanley, WEYL GROUPS, THE HARD LEFSCHETZ THEOREM, AND THE
%SPERNER PROPERTY.}

{\bf Outline.} The paper is organized as follows. We recall some standard definitions in Section \ref{section:background} and the relevant results of \cite{HAN92} in Section \ref{section:hans-results}. We then give a proof of Theorem \ref{theorem:formula-for-dimension} in Section \ref{section:proof-of-theorem} and describe the algorithm for computing the dimension in the general case in Section \ref{section:algorithm}.

{\bf Acknowledgements.} We would like to thank the LoG(M) program (\cite{logm}) at the University of Michigan for providing a platform to run this project. We appreciate the guidance of Ahmad Barhoumi and Sam Hansen. We also appreciate Cheng Meng for pointing out that our Theorem \ref{theorem:formula-for-dimension} follows from \cite[Theorem 3.8]{HAN92}.

\section{Background}\label{section:background}
Let $R=\mathbb{F}[x_1, \ldots, x_n]$ denote the polynomial ring in $n$ variables with coefficients in a field $\mathbb{F}$. We recall the following standard definitions.
\begin{definition}
    \begin{enumerate}
        \item An element $f \in R$ is said to be a \textbf{monomial} if $f=x_1^{i_1}\dots x_n^{i_n}$ for some $i_1,\dots,i_n \in \mathbb{Z}_{\geq 0}$.
        \item The \textbf{degree} of a monomial $x_1^{i_1}\dots x_n^{i_n}$ is defined to be $i_1+\dots + i_n$ and the \textbf{degree} of a polynomial $f = \displaystyle \sum_{(i_1,\dots,i_n) \in (\mathbb{Z}_{\geq 0})^n} a_{i_1,\dots,i_n} x_1^{i_1}\dots x_n^{i_n} \in R$ is defined to be $\max \{ i_1+\dots + i_n \mid a_{i_1,\dots,i_n} \neq 0 \}$.
    \end{enumerate}
\end{definition}

%In this paper, we refer to the \textit{dimension} of a polynomial quotient ring as its vector-space dimension over $\mathbb{F}$.

% Ignoring the multiplication operation between elements, the ring $\mathbb{F}[x_1, \ldots, x_n]$ can also be thought of as a vector space over $\F$, though it is an infinite dimensional one. Likewise, given an ideal $I \subset R$, the quotient ring $R/I$ also has a vector space structure.

% Elements in the quotient ring are no longer individual polynomials, but rather equivalence classes of polynomials with respect to the ideal $I$. The additive identity of the vector space $R/I$ is the equivalence class $0+I$.

We will now review some definitions from commutative algebra.
\begin{definition}[Colon Ideal] \label{colonideal}
    Let $I$ and $J$ be arbitrary ideals in a commutative ring $R$. The colon ideal $(I:J)$ is the set
        $$(I:J) \coloneqq \{ r \in R : rJ \subseteq I \}.$$ 
    To verify that this is an ideal, suppose $a_1, a_2 \in (I:J)$. Then $a_1 j \in I$ for all $j \in J$ and $a_2 j \in I$ for all $j \in J$. Thus $(a_1 + a_2)j = a_1 j + a_2 j \in I$, since $I$ is an ideal and thus closed under addition. Similarly, if $a \in (I:J)$ and $r \in R$, then $(ra)j = r(aj) \in I$ for all $j \in J$, since $I$ is an ideal. 
\end{definition}

\begin{definition}[Short Exact Sequence]
    Let $A,B$ and $C$ be vector spaces, and let $f,g$ be linear transformations $f: A \to B, g: B \to C$ such that $f$ is injective, $g$ is surjective, and $\ker(g) =\im(f)$. Then we say that the sequence
        $$0 \to A \xrightarrow[]{f} B \xrightarrow[]{g} C \to 0$$ 
    is a \textbf{short exact sequence.} Note that the kernel of each map is equal to the image of the previous map. 
\end{definition}

\begin{lemma} \label{SES}
    Let $A,B$ and $C$ be vector spaces such that $$0 \to A \xrightarrow[]{f} B \xrightarrow[]{g} C \to 0$$ is a short exact sequence. Then
        $$ \dim(B) = \dim(A)+\dim(C).$$
    In other words, vector space dimension is additive across short exact sequences. 
\end{lemma}

\begin{proof}
 By the Rank Nullity Theorem, $$\dim(B) = \rank(g) + \dim(\ker(g)).$$ By assumption, $g$ is surjective, so $\rank(g) = \dim(C)$. Furthermore, we also assumed that $\ker(g) = \im(f)$, so $\dim(\ker(g)) = \dim(\im(f)) = \rank(f)$. Since $f$ is injective, another use of the Rank Nullity Theorem tells us that $$\dim(A) = \rank(f) + \dim(\ker(f)) = \rank(f).$$ Thus we have that  
 $$\dim(B) = \dim(A) + \dim(C).$$ 
\end{proof}

%\begin{definition}[Monomial]
%    \hfill \break
%    A \textbf{monomial} is a polynomial that only contains one non-zero term. For example, $x$, and $xyz$ are monomials, but $x+y$ is not. 
%\end{definition} 
%\begin{definition}[Degree]
%    \hfill \break
%    In $n$ variables, we define the \textbf{degree} of a monomial term as the sum of the powers of each $x_i$. Let $a_i \in \F$, $n$ a  nonnegative integer, and $d_i$ positive integers. Then
%    $$\deg \left( \prod_{i=1}^n a_i x_i^{d_i} \right) = \sum_{i=1}^n d_i$$
%   We define the degree of polynomial in $n$ variables as the highest degree of its monomial terms. The degree of the $0$ polynomial is undefined.
%\end{definition}
%    \begin{definition}[Monic]
%    \hfill \break
%    A polynomial is \textbf{monic} if the coefficient of its highest degree term is 1. For examples,  $x+1$ and $y^6+3$ are monic polynomials, but $2x$ is not. 
%\end{definition}

As a precursor to our main result, we will now discuss the vector space dimension of simpler rings: $R/(x_1^{d_1}, \ldots, x_n^{d_n})$ and $R_D$, where $R = \F[x_1, \dots, x_n]$. 

\begin{enumerate}
    \item Observe that the set of all monomials is a basis for $R$ as a vector space over $\mathbb{F}$.
    %For example, the polynomial $3x + 9y + 2z$ is in the span of the monic monomials $x$, $y$, and $z$. Intuitively, every polynomial is the weighted sum of different monic monomials. 
%Consider the monomial ideal $\mathcal{I}=(x_1^{d_1}, \ldots, x_n^{d_n})$ in a polynomial ring $R$, how many monomials are linearly independent in $R/\mathcal{I}$?
%From its definition, any monic monomial is just a product of some variables raised to some powers.
%If we consider an arbitrary monomial, each $x_i$ term can have degree in $\{0, 1, \ldots, d_i-1\}$, a total of $d_i$ choices. Any higher degree and the term would be in the ideal and thus congruent to $0 + \mathcal{I}$ in $R/\mathcal{I}$.
Similarly, we see that the set of monomials $S = \{ x_1^{i_1}\dots x_n^{i_n} \mid 0 \leq i_j \leq d_j-1 \text{ for every }j \}$ is a basis for $R/(x_1^{d_1}, \ldots, x_n^{d_n})$. The number of elements of $S$ is given by $\prod_{i=1}^n d_i$ since for an arbitrary element $x_1^{i_1}\dots x_n^{i_n}$ of $S$, $i_j$ has $d_j$ choices for each $j$. This gives
\[\dim_{\mathbb{F}} R/(x_1^{d_1}, \ldots, x_n^{d_n}) = \prod_{i=1}^n d_i. \]
    \item To calculate $\dim_{\mathbb{F}}(R_D)$,
    %Another natural question one may ask is: what is the vector space dimension of the set of all polynomials of a particular degree, say $D$? 
        we count the number of monomials $x_1^{d_1} \cdots x_n^{d_n}$ of degree $D$, which is equivalent to finding the number of non-negative integer solutions to:
$$d_1 + \ldots + d_n = D \quad \text{such that} \quad 0 \leq d_i$$
We use a ``stars and bars'' argument, in which we think of $D$, the total number of indeterminates in the monomial, as the number of stars, and the bars divide the indeterminates into $n$ categories, one for each variable $x_1, \dots, x_n$.  We then see that
$$\dim_\F(R_D) = \binom{D+n-1}{n-1}.$$
\end{enumerate}

%Putting these observations together yields some basic, preliminary formulas:
%\begin{center}
%    \begin{tabular}{c|c}
%        Vector Space & Dimension (counted by finding a basis of monic monomials) \\ \hline
%        $R/(x_1^{d_1}, \dots, x_n^{d_n})$ & $\ds{\prod_{i=1}^n d_i}$ \\ \hline
%        $R_D$ & $\ds{\binom{D+n-1}{n-1}}$
%    \end{tabular}
%\end{center}

\section{Han's results}\label{section:hans-results}
We now revisit the quotient ring of interest $R/I$ where $R = \mathbb{F}_2[x, y, z]$ and $I = (x^{d_1}, y^{d_2}, z^{d_3}, x+y+z)$.
%\begin{align*}
%    R & = \mathbb{F}_2[x, y, z] \\
%    I & = (x^{d_1}, y^{d_2}, z^{d_3}, x+y+z) \\
%    R/I & = \mathbb{F}_2[x, y, z]/(x^{d_1}, y^{d_2}, z^{d_3}, x + y + z)
%\end{align*}
Notice that $I$ is not a monomial ideal and so, unlike in the previous section, nonzero monomials in $R/I$ don't necessarily form a basis for $R/I$ over $\F$. For example, the elements $\{x+I,y+I,z+I \}$ do not form a linearly independent set since $x+y+z+I=0+I$. So, counting the vector space dimension of $R/I$ is more complicated.

We now construct an isomorphism of $R/I$ to a polynomial ring in one fewer variable. First, we define notation that will be useful.
\begin{notation}
    Let $R$ be a commutative ring, $I$ an ideal of $R$, and $r$ and element of $R$. When the ideal $I$ is clear by context, we use $\overline{r}$ to denote the congruence class $r+I \in R/I$. 
\end{notation}

Consider the map:
\begin{align*}
    \phi: \mathbb{F}_2[x, y, z]/(x^{d_1}, y^{d_2}, z^{d_3}, x + y + z) & \to \mathbb{F}_2[x, y]/(x^{d_1}, y^{d_2}, (x+y)^{d_3}) \\
    \varphi(\overline{f(x,y,z))}) &= \overline{f(x,y,(x+y))}
\end{align*}
We can easily check that $\varphi$ is well defined and a ring homomorphism. In fact, $\varphi$ is an isomorphism since we can define the inverse map as follows:
\begin{align*}
    \varphi^{-1} : \mathbb{F}_2[x, y]/(x^{d_1}, y^{d_2}, (x+y)^{d_3}) & \to \mathbb{F}_2[x, y, z]/(x^{d_1}, y^{d_2}, z^{d_3}, x + y + z) \\
%  \varphi^{-1}( \overline{1}) &= \overline{1} \\
%    \varphi^{-1}(\overline{x}) &= \overline{x} \\
%    \varphi^{-1}(\overline{y}) &= \overline{y}
    \varphi^{-1}(\overline{f(x,y)}) &= \overline{f(x,y)}.
\end{align*}
%and extending so that $\varphi^{-1}$ is a homomorphism, i.e., $$\varphi^{-1}(\overline{f(x,y)}) = \overline{f(x,y)}.$$ 

\noindent Let us check that $\varphi^{-1}$ is well-defined. Suppose that  $f(x,y)-g(x,y) \in (x^{d_1}, y^{d_2}, (x+y)^{d_3}))$. Then $f(x,y)-g(x,y)=ax^{d_1}+by^{d_2} + c(x+y)^{d_3}$ for some $a,b,c \in \mathbb{F}_2[x,y]$. Note also that in $\mathbb{F}_2[x,y,z]/(x^{d_1}$,$y^{d_2}, z^{d_3}, (x+y+z))$, the congruence classes $z + (x^{d_1}, y^{d_2}, z^{d_3}, x + y + z)$ and $x+y+ (x^{d_1}, y^{d_2}, z^{d_3}, x + y + z)$ are equal. Thus,  $\overline{(x+y)^{d_3}} = \overline{z^{d_3}}$. Therefore, $ax^{d_1}+by^{d_2} + c(x+y)^{d_3} \in (x^{d_1}, y^{d_2}, z^{d_3},(x+y+z)) $. Thus $\overline{f(x,y)} = \overline{g(x,y)}$ in $\mathbb{F}_2[x,y,z]/(x^{d_1}, y^{d_2}, z^{d_3},(x+y+z))$, so the map is well-defined.

\noindent Let us now check that $\varphi \circ \varphi^{-1}$ and $\varphi^{-1} \circ \varphi$ are the identity maps on $\mathbb{F}_2[x,y]/(x^{d_1}, y^{d_2},(x+y)^{d_3})$ and $\mathbb{F}_2[x,y,z]/(x^{d_1}, y^{d_2}, z^{d_3},(x+y+z))$, respectively. For any $f \in \F_2[x,y]$. we have
\begin{align*}
    \varphi (\varphi^{-1} (\overline{f(x,y})) &= \varphi(\overline{f(x,y)}) \\ 
    &= \overline{f(x,y)}
\end{align*}
since $f(x,y)$ does not contain $z$. Conversely, for any $f(x,y,z) \in \F_2[x,y,z]$, we have
\begin{align*}
    \varphi^{-1}(\varphi(\overline{f(x,y,z))}) &= \varphi^{-1}(\overline{f(x,y,x+y)}) \\ 
    &= \overline{f(x,y,x+y)} \\
    &= \overline{f(x,y,z))}
\end{align*}
since $ \overline{x+y} = \overline{z}$ in $\mathbb{F}_2[x,y,z]/(x^{d_1}, y^{d_2}, z^{d_3}, x+y+z)$. Thus we have an isomorphism:
\[ \mathbb{F}_2[x, y, z]/(x^{d_1}, y^{d_2}, z^{d_3}, x + y + z)  \xrightarrow{\sim} \mathbb{F}_2[x, y]/(x^{d_1}, y^{d_2}, (x+y)^{d_3}) \]
In particular, we have:
\begin{equation}\label{equation:isomorphism-2-variables}
    \dim_{\F_2}\mathbb{F}_2[x, y, z]/(x^{d_1}, y^{d_2}, z^{d_3}, x + y + z) = \dim_{\F_2}\mathbb{F}_2[x,y]/(x^{d_1}, y^{d_2}, (x+y)^{d_3}),
\end{equation}
Also, we note that the LHS above is symmetric in $d_1,d_2,d_3$ and so, the RHS is as well. Therefore we may and will assume that $d_1 \leq d_2 \leq d_3$.

% Note that in the quotient ring $\mathbb{F}_2[x, y, z]/(x^{d_1}, y^{d_2}, z^{d_3}, x + y + z)$, the elements $\overline{z} = -\overline{x+y}$. Therefore in $\mathbb{F}_2[x, y, z]/(x^{d_1}, y^{d_2}, z^{d_3}, x + y + z)$, 
%  $$ \overline{ ax^{d_1}+by^{d_2} + c(x+y)^{d_3}} = \overline{ax^{d_1}+by^{d_2} + c(z)^{d_3}}.$$ 
%Verifying that $\varphi \circ \varphi^{-1}$ is the identity map is also straightforward.

%The following approach focuses on not calculating the overall dimension outright, but rather reducing $d_1, d_2, d_3$ bit by bit, collecting dimension as we go.

We now state two results from \cite{HAN92}, specialized to the particular case $\F = \mathbb{F}_2$, which tell us how to calculate $\dim_{\F_2}\mathbb{F}_2[x,y]/(x^{d_1}, y^{d_2}, (x+y)^{d_3})$ in certain cases.

\begin{proposition}\cite[Proposition 1.6 (i)]{HAN92} \label{corollary:han-result-power-of-2}
    \hfill \break
    Suppose that $d_1, d_2 \leq q$, and $d_3 \geq q$ where $q = 2^e$ is some power of 2. Then:
    \begin{align*} 
        \dim_{\mathbb{F}_2} \mathbb{F}_2[x, y]/(x^{d_1}, y^{d_2}, (x+y)^{d_3})
        & = \dim_{\mathbb{F}_2} \mathbb{F}_2[x, y]/(x^{d_1}, y^{d_2}) \\
        & = d_1 \cdot d_2
    \end{align*}
\end{proposition}

% \begin{proof}
%     Since $d_3 \geq q$ we can write:
%     \begin{align*}
%         (x+y)^{d_3} & = (x+y)^{d_3 - q + q} \\
%         & = (x+y)^{d_3 - q} \cdot (x+y)^q \\
%         & = (x+y)^{d_3 - q} \cdot (x+y)^{2^e}
%     \end{align*}
%     Because 2 divides $\binom{2^e}{i}$ for all $1 \leq i < 2^e$, the middle terms in the binomial expansion of $(x+y)^{2^e}$ will be divisible by 2 and are zero in $\mathbb{F}_2$. Therefore we are left with the freshman's dream:
%     \begin{align*}
%         & = (x+y)^{d_3 - q} \cdot (x+y)^{2^e} \\
%         & = (x+y)^{d_3 - q} \cdot (x^{2^e} + y^{2^e}) \\
%         & = (x+y)^{d_3 - q} \cdot (x^q + y^q)
%     \end{align*}
%     As $d_1, d_2 \leq q$, the term on the right is in the ideal $(x^{d_1}, y^{d_2})$. This implies that the ideal $(x^{d_1}, y^{d_2}, (x+y)^q)$ is equal to the ideal $(x^{d_1}, y^{d_2})$, so 
%     $$\dim \mathbb{F}_2[x, y]/(x^{d_1}, y^{d_2}) = d_1 \cdot d_2 $$
% \end{proof}

\begin{proposition}
    \cite[Proposition 1.6 (ii)]{HAN92} \label{cor:han-reduction}
    \hfill \break
    Suppose that $d_1 \leq q$, and $d_2, d_3 \geq q$ where $q = 2^e$ is some power of 2. Then:
    \begin{align*} 
        \dim_{\mathbb{F}_2} \mathbb{F}_2[x, y]/(x^{d_1}, y^{d_2}, (x+y)^{d_3}) = \: & d_1 \cdot q \\
        + \: & \dim_{\mathbb{F}_2} \mathbb{F}_2[x, y]/(x^{d_1}, y^{d_2-q}, (x+y)^{d_3 - q})
    \end{align*}
\end{proposition}

\section{Proof of Theorem \ref{theorem:formula-for-dimension}}\label{section:proof-of-theorem}

In this section, we provide a proof of Theorem \ref{theorem:formula-for-dimension}, which recovers a result of \cite[Theorem 3.8]{HAN92}. Let us recall the statement.

\majorThm* 

\noindent By Equation \ref{equation:isomorphism-2-variables}, it suffices to show that 
$$\dim_{\F_2}\frac{\mathbb{F}_2[x,y]}{(x^{d_1}, y^{d_2}, (x+y)^{d_3})} = d_1d_2 + d_2d_3 + d_1d_3 -q(d_1+d_2+d_3) +q^2.$$
To prove Theorem \ref{theorem:formula-for-dimension}, we first prove the following lemmas.

\begin{lemma}\label{lemma:dimension-of-successive-quotients}
    Let $q$ be a power of $2$ and let $\frac{q}{2}< d_1
    \leq d_2 \leq  q$. 
    For any $\frac{q}{2}< t < q$, define:
\begin{align*}
    J_t := (x^{d_1},y^{d_2},(x+y)^{t}).
\end{align*}
    Then:
\[ \dim \frac{J_t}{J_{t+1}} = d_1 + d_2 - q. \]
\end{lemma}
\begin{proof}
Consider $S = \{ x^i(x+y)^{t} \mid 0 \leq i \leq d_1+d_2-q-1 \}$. We will prove in two steps that $S$ is a basis of $\frac{J_t}{J_{t+1}}$, thereby proving the lemma.

\textbf{Step 1:} Let us show that $S$ is a spanning set of $\frac{J_t}{J_{t+1}}$. First observe that:
\begin{align*}
    (x+y)^{t+1} &= (x+y)(x+y)^{t}\\
    &= x(x+y)^{t} + y(x+y)^{t}\\
    \implies x(x+y)^{t} &= y(x+y)^{t} \text{ in } \frac{J_t}{J_{t+1}}.
\end{align*}
Thus $\frac{J_t}{J_{t+1}}$ is spanned by the set $\{x^i(x+y)^t \mid i \geq 0 \}$. Now, if we have $x^{i}(x+y)^{t}$ for any $i \geq d_1+d_2-q$, writing $i=i_1+i_2$ with $i_1 \geq d_1-\frac{q}{2}$ and $i_2 \geq d_2-\frac{q}{2}$, we get:
\begin{align*}
    x^{i}(x+y)^{t} &= x^{i_1} y^{i_2} (x+y)^{\frac{q}{2}}(x+y)^{t-\frac{q}{2}}\\
    &= x^{i_1} y^{i_2} (x^{\frac{q}{2}}+y^{\frac{q}{2}})(x+y)^{t-\frac{q}{2}}\\
    &= (x^{\frac{q}{2}+i_1}y^{i_2}+x^{i_1}y^{\frac{q}{2}+i_2})(x+y)^{t-\frac{q}{2}}\\
    &= 0
\end{align*}
since $x^{d_1} = 0 = y^{d_2} \text{ in } \frac{J_t}{J_{t+1}}$. This finishes Step 1.

\textbf{Step 2:} Let us show that $S$ is linearly independent. Assume for the sake of contradiction that there exists:
\[  \sum_{i=0}^{d_1+d_2-q-1} \lambda_i x^i (x+y)^t = 0 \text{ in } \frac{J_t}{J_{t+1}} \]
with $\lambda_a \neq 0$ (equivalently, $\lambda_a = 1$) for some $0 \leq a \leq d_1+d_2 - q - 1$. We get:
\begin{align*}
    &\sum_{i=0}^{d_1+d_2-q-1} \lambda_i x^i (x+y)^t = x^{d_1}f + y^{d_2}g + (x+y)^{t+1}h \text{ in } \mathbb{F}_2[x,y]\\
     \implies &(x+y)^t \left( \sum_{i=0}^{d_1+d_2-q-1} \lambda_i x^i - (x+y)h \right)  = x^{d_1}f + y^{d_2}g.
\end{align*}
for some $f,g,h \in \mathbb{F}_2[x,y]$. Now  look at this equality in degree $t+a$:
\[ (x+y)^t \left( x^a - (x+y)\widetilde{h} \right)  = x^{d_1}\widetilde{f} + y^{d_2}\widetilde{g} \]
where $\widetilde{f},\widetilde{g},\widetilde{h}$ are homogeneous polynomials. Multiply by $(x+y)^{q-1-t}$ on both sides to get:
\begin{equation} \label{equation:linear-dependence}
    (x+y)^{q-1} \left( x^a - (x+y)\widetilde{h} \right)  = x^{d_1}\left[ (x+y)^{q-1-t}\widetilde{f} \right] + y^{d_2}\left[(x+y)^{q-1-t}\widetilde{g} \right].
\end{equation}

We will now produce a term $x^{b_1}y^{b_2}$ in the LHS such that $b_1 < d_1$ and $b_2 < d_2$. This would be a contradiction since every monomial $x^{a_1}y^{a_2}$ appearing in the RHS of (\ref{equation:linear-dependence}) has either $a_1\geq d_1$ or $a_2 \geq d_2$. By Lemma \ref{lemma-characterization-of-polynomials} below, we can write
\begin{align*}
    x^a - (x+y)\widetilde{h} = \sum_{i=0}^a \mu_ix^iy^{a-i}
\end{align*}
with $\mu_i$ satisfying $\displaystyle \sum_{i=0}^a \mu_i = 1$. Therefore we get:
\begin{align*}
    (x+y)^{q-1} \left( \sum_{i=0}^a \mu_ix^iy^{a-i} \right) &= \left( \sum_{i=0}^{q-1} x^i y^{q-1-i} \right) \left( \sum_{i=0}^a \mu_ix^iy^{a-i} \right) \\
    &= \sum_{i=0}^{q-1+a} \left( \sum_{k=\max(i-(q-1),0)}^{\min(i,a)} \mu_k \right) x^iy^{q-1+a-i}.
\end{align*}
Now, consider the term $x^{d_1-1}y^{q+a - d_1}$. Since $d_1-1<q-1$, we have that $\max(d_1-1 -(q-1),0) = 0$, and since $d_1 -1 \geq d_1-1+ d_2-q \geq a$, we have that $\min(d_1-1,a) = a$. Therefore the coefficient of $x^{d_1-1}y^{q+a - d_1}$ is just $\displaystyle \sum_{k=0}^{a} \mu_i = 1$. This implies that the term $x^{d_1-1}y^{q+a - d_1}$ survives in the LHS of (\ref{equation:linear-dependence}) and it is indeed of the form $x^{b_1}y^{b_2}$ such that $b_1 < d_1$ and $b_2 < d_2$, since $q+a-d_1 < d_2$. This finishes the proof. 
\end{proof}

\begin{lemma}\label{lemma-characterization-of-polynomials}
    Let $S_d$ denote the set of all homogeneous degree $d$ polynomials in $\mathbb{F}_2[x,y]$. Consider the following injective map:
    \begin{align*}
        \varphi_a: S_{a-1} &\hookrightarrow S_a\\
        p(x,y) &\mapsto x^a - p(x,y)(x+y).
    \end{align*}
    Then the image is exactly the subset
    \[ \left\{ \sum_{i=0}^a \mu_ix^iy^{a-i} \mid \displaystyle\sum_{i=0}^a \mu_i = 1 \right\}. \]
\end{lemma}
\begin{proof}
    Note that $S_a$ is a vector space over $\mathbb{F}_2$ of dimension $a+1$. Consider the linear map given by evaluating at $(1,1)$:
    \begin{align*}
        \mathrm{ev} : S_a &\to \mathbb{F}_2\\
        \sum_{i=0}^a \mu_ix^iy^{a-i} =: r(x,y) &\mapsto r(1,1) = \sum_{i=0}^a \mu_i.
    \end{align*}
    Observe:
    \[ \mathrm{ev}^{-1}(1) = \{ \sum_{i=0}^a \mu_ix^iy^{a-i} \mid \sum_{i=0}^a \mu_i = 1\}. \]
    The set $\mathrm{ev}^{-1}(1)$ has $2^a$ elements since it is a translate of the kernel of $\mathrm{ev}$. Now, for any $p(x,y) \in S_{a-1}$, we have that $x^a - p(x,y)(x+y) \in \mathrm{ev}^{-1}(1)$ since its value at $(1,1)$ is $1$. Thus $\mathrm{Im} (\varphi_a) \subset \mathrm{ev}^{-1}(1)$. But since both sets have $2^a$ elements, we get $\mathrm{Im} (\varphi_a) = \mathrm{ev}^{-1}(1)$ as required.
\end{proof}

We can now finally prove Theorem \ref{theorem:formula-for-dimension}.
\begin{proof}[Proof of Theorem \ref{theorem:formula-for-dimension}]
    We will prove the theorem by descending induction on $d_3$ with the base case being $d_3 = q$. Denote:
    \[ f(d_1,d_2,d_3) = d_1d_2 + d_2d_3 + d_1d_3 -q(d_1+d_2+d_3) +q^2. \]
    When $d_3=q$, Corollary \ref{corollary:han-result-power-of-2} gives:
    \[ \dim_{\mathbb{F}_2} \frac{\mathbb{F}_2[x,y]}{(x^{d_1},y^{d_2},(x+y)^{q})} = d_1 d_2, \] 
    while
    \begin{align*}
        f(d_1,d_2,q) &= d_1d_2 + d_2q + qd_1 -q(d_1+d_2+q) +q^2\\
        &= d_1d_2,
    \end{align*}
    which proves the base case of the induction.
    
    \noindent Now, fix an $N \in \mathbb{N}$ with $\frac{q}{2}<N < q$ and assume by induction that the result is true for all $d_3$ with $N < d_3 \leq q$. Let us prove the result for $d_3=N$. With notation as in Lemma \ref{lemma:dimension-of-successive-quotients}, we have:
    \[ \dim_{\mathbb{F}_2} \frac{\mathbb{F}_2[x,y]}{J_{N+1}} = d_1d_2 + d_2(N+1) + (N+1)d_1 -q(d_1+d_2+N+1) +q^2.\]
    We have the exact sequence:
    \[ 0 \to \frac{J_N}{J_{N+1}} \to \frac{\mathbb{F}_2[x,y]}{J_{N+1}} \to \frac{\mathbb{F}_2[x,y]}{J_{N}} \to 0.  \]
    By Lemma \ref{lemma:dimension-of-successive-quotients}, we have that:
        \[ \dim_{\mathbb{F}_2} \frac{J_N}{J_{N+1}} = d_1 + d_2 - q. \]
    Putting all of these together, we get:
    \begin{align*}
        \dim_{\mathbb{F}_2} \frac{\mathbb{F}_2[x,y]}{J_{N}} = &\dim_{\mathbb{F}_2} \frac{\mathbb{F}_2[x,y]}{J_{N+1}} - \dim_{\mathbb{F}_2} \frac{J_N}{J_{N+1}}\\
        = &\left(d_1d_2 + d_2(N+1) + (N+1)d_1 -q(d_1+d_2+N+1) +q^2 \right)\\
        &- (d_1 + d_2 - q)\\
        = &d_1d_2 + d_2N + Nd_1 -q(d_1+d_2+N) +q^2\\
        = &f(d_1,d_2,N).
    \end{align*}
    This finishes the proof.
\end{proof}

\section{Algorithm for calculating $\dim \mathbb{F}_2[x, y, z]/(x^{d_1}, y^{d_2}, z^{d_3}, x + y + z)$}\label{section:algorithm}
%We have shown two conjectures regarding special cases of $d_1, d_2, d_3$. 
%The conjectures cover the last missing case regarding Han's reductions.
We now give an algorithm to calculate the dimension of $\mathbb{F}_2[x, y, z]/(x^{d_1}, y^{d_2}, z^{d_3}, x + y + z)$.

Let $d_1,d_2,d_3$ be positive integers. If there exists a $q=2^e$ such that $\frac{q}{2} \leq d_1,d_2,d_3 \leq q$ or $d_1,d_2 \leq q \leq d_3$, then $\dim_{\mathbb{F}_2}(d_1,d_2,d_3))$ is given by Theorem \ref{theorem:formula-for-dimension} or Corollary \ref{corollary:han-result-power-of-2}, respectively. Otherwise, it must be that there is some $q=2^e$ such that $\frac{q}{2} < d_1 \leq q \leq d_2,d_3$. Then Corollary \ref{cor:han-reduction} allows us to reduce to the case of $d_1, d_2-q$ and $d_3-q$. If $d_2-q, d_3-q \leq \frac{q}{2}$, we are again in the case of Corollary \ref{corollary:han-result-power-of-2}; if $ \frac{q}{2} < d_1, d_2-q, d_3-q \leq q$, then we are again in the case of Theorem \ref{theorem:formula-for-dimension}. Else we apply Corollary \ref{cor:han-reduction} again. At every stage we either have a formula for $\mathbb{F}_2[x, y, z]/(x^{d_1}, y^{d_2}, z^{d_3}, x + y + z)$ or a way to decrease two of the exponents. This process must eventually terminate, since $d_1,d_2,d_3$ are positive integers. 

\begin{algorithm}[ht!]
\caption{} \label{algo}
\begin{algorithmic}[1]
    \Require{$d_1, d_2, d_3 \in \Z_{\geq 0}$
    }
    \Ensure{$\dim_{\mathbb{F}_2}[x, y, z]/(x^{d_1}, y^{d_2}, (x+y)^{d_3})$}

    \State $\tt{sum} \leftarrow 0$
    
    \item[]
    \State $i \leftarrow \max 
    \{
        \lfloor \log_2(d_1) \rfloor,
        \lfloor \log_2(d_2) \rfloor,
        \lfloor \log_2(d_3) \rfloor
    \}$ 
    \While {$i \geq 0$}
        \State reorder $d_1, d_2, d_3$ in non-decreasing order
        \State $q \leftarrow 2^i$
        \If {$d_1, d_2, d_3 < q$}
            \State do nothing
        \ElsIf {$d_1, d_2 < q$ and $q \leq d_3$}
            \Comment {Corollary \ref{corollary:han-result-power-of-2}}
            \State $\tt{sum} \leftarrow \tt{sum} + d_1 d_2$
            
            \State \textbf{break}
        \ElsIf {$d_1 < q$ and $q \leq d_2, d_3$}
            \Comment {Corollary \ref{cor:han-reduction}}
            \State $\tt{sum} \leftarrow \tt{sum} + d_1 q$
        \ElsIf {$q \leq d_1, d_2, d_3$}
            \Comment {Theorem \ref{theorem:formula-for-dimension}}
        	\State $ \tt{sum} \leftarrow \tt{sum} + d_1 d_2 + d_1 d_3 + d_2 d_3 - (d_1 + d_2 + d_3)q + q^2$
            \State \textbf{break}
        \EndIf

        \State $i \leftarrow i-1$
    \EndWhile

    \item[]
    \State \textbf{return} \tt{sum}
\end{algorithmic}
\end{algorithm}

\pagebreak

\printbibliography

\end{document}